\documentclass[11pt]{amsart}
\usepackage{amsmath, amssymb, amsthm}
\usepackage{enumerate}
\newtheorem*{lemma}{Lemma}

\newtheorem*{theoremb}{Representation Formulas}

\newcommand{\R}{{\mathbb{R}}}

\begin{document}
\title{Bessel functions and the wave equation}
\author{Alberto Torchinsky}
\date{}
\begin{abstract}
We  solve the Cauchy problem for the $n$-dimensional wave equation 
using elementary properties of the  Bessel functions.
\end{abstract}
\maketitle

With
  $\nabla^2=D^2_{x_1x_1}+\cdots + D^2_{x_nx_n}$  the Laplacian in $\R^n$, where
 \[ D^2_{x_kx_k}=\frac{\partial^2}{\partial x_k^2}\,,\quad 1\le k\le n\,,
 \]
and $D_t$ and $ D_{tt}$ indicating the first and second order derivatives
 with respect to the variable $t\in\R$, respectively,  the
  wave equation in the upper half-space $\R^{n+1}_+$ is given by
\begin{equation}D^2_{tt}u(x,t)=\nabla^2 u(x,t)\,, \quad x\in \R^n, t>0\,,
\end{equation}
and  the Cauchy problem for this equation consists of finding
  $u(x,t)$ that satisfies (1)  subject to the initial conditions
  \[u(x,0)=\varphi(x)\quad {\text{and}}\quad  D_tu(x,0)=\psi(x)\,,\quad x\in\R^n ,
  \]
where for simplicity we shall take  $\varphi$ and $\psi$ in  $\mathcal S(\R^n)$. 

Applying the Fourier transform to (1) in the
space variables,  considering $t$ as a parameter,  it
readily follows that $\widehat {\nabla^2
u}(\xi,t)=-|\xi|^2\,\widehat u(\xi,t)$, and so $\widehat{u}$
satisfies
\[ D^2_{tt}\widehat{u}(\xi,t) +|\xi|^2 \widehat{u}(\xi,t)=0\,,\quad \xi\in \R^n, t>0\,,\]
subject to
\[ \widehat{u}(\xi,0)=\widehat{\varphi}(\xi)\quad {\text{and}}\quad
 D_t \widehat u(\xi,0)=\widehat{\psi}(\xi)\,,\quad \xi\in\R^n \,.\]

 For each fixed $\xi\in\R^n$  this resulting ordinary differential equation in $t$ is
 the simple harmonic oscillator equation with constant angular frequency $|\xi|$,
and so
\[
\widehat{u}(\xi,t)=\widehat{\varphi}(\xi) \cos(t|\xi|) +
{\widehat{\psi}(\xi)} \,\frac{ \sin(t|\xi|)}{|\xi|}\,,\quad \xi\in \R^n, t>0\,.\]
Hence, the Fourier inversion formula gives for
$(x,t)\in\R^{n+1}_+$,
\begin{align}
u(x,t)=\frac{1}{(2\pi)^n}\int_{\R^n} \widehat{\varphi}(\xi)
 &\cos(t|\xi|)\, e^{i\xi\cdot x}\,d\xi\\
 &+\frac{1}{(2\pi)^n}\int_{\R^n}
{\widehat{\psi}(\xi)} \frac{ \sin(t|\xi|)}{|\xi|}\ e^{i\xi\cdot
x}\,d\xi\,.\nonumber
\end{align}

Since the first integral in (2) can
be obtained from the second by differentiating with respect
to $t$, we will concentrate on the latter.
The idea  is  to interpret   $\sin(|\xi|t)/|\xi| $ as
the Fourier transform of a   tempered distribution, and the key ingredient for this are the following
representation formulas established in \cite{WT}.

\begin{theoremb} Assume that  $n$ is  an odd integer greater than or equal
to $3$. Then, with $d\sigma$ the element of surface area on $\partial B(0,R)$,
\begin{equation}
\frac{\sin (R|\xi|)}{|\xi|}= c_n\,\Big(\frac1{R}\,\frac{\partial}{\partial
R}\Big)^{(n-3)/2}
\Big( \frac{1}{\,\omega_n\,R}\int_{\partial B(0,R)}e^{-ix\cdot \xi}\,d\sigma(x)\Big)\,,
\end{equation}
where $R>0$, $\omega_n$ is the surface measure of the unit ball in
$\R^n$, and
$c_n^{-1}=(n-2)(n-4)\cdots 1 $.

On the other hand, if $n$ is an even integer greater than or equal
to $2$,
\begin{equation}
\frac{\sin(R|\xi|)}{|\xi|}= d_n\,\Big(\frac1{R}\,\frac{\partial}{\partial
R}\Big)^{(n-2)/2}
\Big( \frac{1}{\,v_n }\int_{  B(0,R)}\frac{1}{\sqrt{R^2-|x|^2}}\,
e^{-ix\cdot \xi}\,d x\Big)\,,
\end{equation}
where $R>0$,  $d_n^{-1}={n(n-2)(n-4)\cdots 2 }$,
and $v_n$ is the volume of the unit ball in $\R^n$.
\end{theoremb}

The purpose of this note is to establish (3) and (4) using elementary properties of Bessel functions.
$J_\nu(x)$, the Bessel function of order $\nu$,  
 is defined as the solution of the second order linear equation
\[ x^2\, \frac{d^2y}{dx^2} + x \frac{dy}{dx}+(x^2-\nu^2)\,y=0\,. \]

Several basic properties of the Bessel functions follow readily from their power series    expression \cite {W}. They include the recurrence formula 
\begin{equation}\frac{d}{dx}(x^\nu J_\nu(x))=x^\nu J_{\nu-1}(x)\,,
\end{equation}
the integral representation of Poisson type
\begin{equation} J_\nu(x)=\frac{(x/2)^\nu}{\Gamma(\nu+1/2)\Gamma(1/2)}\,\int_{-1}^1 (1-s^2)^{\nu-1/2}\,e^{ixs}\, ds\,,\end{equation}
and the identity
\begin{equation}J_{1/2}(x)=\frac{\sqrt 2}{\sqrt \pi}\frac1{x^{1/2}}\sin(x)\,,
\end{equation}
for $x>0$.

We will consider the odd dimensional case first. The dimensional constant $c_n$ may vary from appearance to appearance until it is finally determined at the end of the proof. To begin recall that for $n \ge 3$, as established in (18) in \cite{WT},
 
\begin{equation*}\frac1{\omega_{n-1} R }\,\int_{\partial B(0,R)}e^{-ix\cdot \xi}\,d\sigma (x) 
= R^{n-2}\int_{-1}^1 e^{iR|\xi|s}\,(1-s^2)^{(n-3)/2}\,ds\,,
\end{equation*}
which combined with (6) above with $\nu-1/2=(n-3)/2$ there, i.e., $\nu=(n-2)/2$, gives
\begin{align*}
\frac1{\omega_{n } R }\,\int_{\partial B(0,R)}e^{-ix\cdot \xi}\,d\sigma (x) &= c_n \,
 R^{n-2}\frac{J_{(n-2)/2}(R |\xi|)}{(R |\xi|)^{(n-2)/2}} \\
&= c_n\,\frac1{|\xi|^{n-2}}\,(R |\xi|)^{(n-2)/2}J_{(n-2)/2}(R|\xi|)\,.
 \end{align*}

Now, by (5) we obtain that
\[\frac{\partial}{\partial R }\Big(\frac1{\omega_{n } R }\,\int_{\partial B(0,R)}e^{-ix\cdot \xi}\,d\sigma (x)\Big)= c_n \,
 \frac1{|\xi|^{n-2}}|\xi|\,\,(R |\xi|)^{(n-2)/2}J_{(n-4)/2}(R|\xi|)\,,
\]
or
\[\frac1{R}\frac{\partial}{\partial R }\Big(\frac1{\omega_{n } R }\,\int_{\partial B(0,R)}e^{-ix\cdot \xi}\,d\sigma (x)\Big)= c_n \,
 \frac1{|\xi|^{n-4}} \,(R |\xi|)^{(n-4)/2}J_{(n-4)/2}(R|\xi|)\,.
\]

Thus, applying the above reasoning $(n-3)/2$ times, (7)
gives
\begin{align*}\Big(\frac1{R}\frac{\partial}{\partial R }\Big)^{(n-3)/2}
\Big(\frac1{\omega_{n } R }\,&\int_{\partial B(0,R)}e^{-ix\cdot \xi}\,d\sigma (x)\Big) \\
&= c_n \,
 \frac1{|\xi| } \,(R |\xi|)^{1/2}J_{1/2}(R|\xi|)\\
 &= c_n \,\frac1{|\xi| }
  \,(R |\xi|)^{1/2}\,\frac{\sin(R|\xi|)}{(R|\xi|)^{1/2}}\\
  &=c_n \,
   \frac{\sin(R|\xi|)}{ |\xi| }  \,.
\end{align*}

The value of $c_n$ is readily obtained as in \cite{WT}, and (3) has been established. 

To consider the case $n$ even,  one generally proceeds at this point  by a reasoning akin to  Hadamard's 
method of descent, i.e.,  the desired result for the wave equation in even dimension $n$ is 
derived from the result in odd dimension $n+1$, as is done for instance in \cite{WT} for the representation formulas. On
the other hand, 
Bessel functions  provide the desired result for the wave equation in even dimensions directly, by a method akin to ascent: the result for the wave equation for dimension $n=2$ is obtained explicitly, and for even dimension $n +2$ is obtained from the result in even  dimension $n$.


We will first prove a preliminary result. The dimensional constant $d_n$ may vary from appearance to appearance until it is finally determined at the end of the proof.
 
\begin{lemma} The following three  statements hold.
\begin{equation}\int_0^\infty \sin(R\rho)\,J_0(t\,\rho)\,d\rho=\frac1{\sqrt{R^2-t^2}} \,H(R-t)\,,\quad R, t>0\,,
\end{equation}
where $H $ denotes the Heavyside  function.

Furthermore, for $\nu\ge 1$,
\begin{equation} \Big( \frac1{R} \frac{\partial}{\partial R }\Big) \Big( 
\int_0^\infty \sin(R\rho)\,\rho^{\nu -1} J_{\nu-1}(t\,\rho)\,d\rho
\Big) =\frac1{t}
\int_0^\infty \sin(R \rho)\, \rho^\nu J_\nu(t \rho)\,d\rho\,,
\end{equation}
and, consequently, for $1\le j\le \nu$, 
\begin{equation} \Big( \frac1{R} \frac{\partial}{\partial R }\Big)^j \Big( 
\int_0^\infty \sin(R\rho)\,\rho^{\nu -j} J_{\nu-j}(t\,\rho)\,d\rho
\Big) =\frac1{t^j}
\int_0^\infty \sin(R \rho)\, \rho^\nu J_\nu(t \rho)\,d\rho\,.
\end{equation}

\end{lemma} 
\begin{proof}
(8) is Formula (6) in \cite{W},  page 405.

Now, 
\begin{align*} \frac{\partial}{\partial R}\Big( \int_0^\infty \sin(R\rho)\,\rho^{\nu-1} 
J_{\nu-1}(t\rho)\,d\rho\Big)
 & = - \int_0^\infty \cos(R\rho)\,\rho^{\nu} J_{\nu-1}(t\rho)\,d\rho\\
 &= - \frac1{t^{\nu}}
\int_0^\infty \cos(R\rho)\,(t \rho)^{\nu} J_{\nu-1}(t\rho)\,d\rho\,,
\end{align*}
which, by (1), equals
\begin{equation*}  \frac{-1}{t^{\nu+1}}
\int_0^\infty \cos(R\rho)\,\frac{\partial}{\partial \rho}\big( \big(t \rho)^{\nu} J_{\nu}(t\rho)\big)\,d\rho
= \frac{R}{t}
\int_0^\infty \sin(R\rho)\,  \rho^{\nu} J_{\nu}(t\rho)\,d\rho\,,
\end{equation*}
which proves (9).

(10) follows by repeated applications of (9), and we have finished.
\end{proof}

Finally, recall that  the Fourier transform of a radial function $f$ on $\R^n$ is given by the expression  \cite {W},
\[\widehat f(|\xi|)=d_n\,\frac1{|\xi|^{(n-2)/2}}\,\int_0^\infty \rho^{n/2}\, f(\rho) \,J_{(n-2)/2}(|\xi|\rho)\,d\rho\,.
\]

In particular,
  we have
\begin{equation}\int_{\R^n} \frac{\sin(R|\xi|)}{|\xi|}\,e^{-ix\cdot\xi}\,d\xi=d_n\,\frac1{|x|^{(n-2)/2}}\,
\int_0^\infty \rho^{n/2} \, \frac{\sin(R \rho)}{\rho}\,J_{(n-2)/2}(|x|\rho)\,d\rho\,.
\end{equation}

Let now $n=2k$ be an even integer. Then by (11), 
  \begin{equation*}\int_{\R^n} \frac{\sin(R|\xi|)}{|\xi|}\,e^{-ix\cdot\xi}\,d\xi=d_n\,\frac1{|x|^{(k-1)}}\,
\int_0^\infty 
 \sin(R \rho)\,\rho^{k-1}  J_{k-1}(|x|\rho)\,d\rho\,,
\end{equation*} 
and, therefore,  (10) with $\nu=j=k-1$ there yields
\begin{align*} \int_{\R^n} \frac{\sin(R|\xi|)}{|\xi|}\,e^{-ix\cdot\xi}\,d\xi &=d_n\, \frac1{|x|^{(k-1)}}
\int_0^\infty \sin(R \rho)\, \rho^{(k-1)} J_{k-1}(|x| \rho)\,d\rho\\
&= d_n\,\Big( \frac1{R} \frac{\partial}{\partial R }\Big)^{(k-1)} \Big( 
\int_0^\infty \sin(R\rho)\, J_{0}(|x|\,\rho)\,d\rho
\Big)\\
&= d_n\,\Big( \frac1{R} \frac{\partial}{\partial R }\Big)^{(k-1)} 
\Big(\frac1{\sqrt{R^2-|x|^2}} \,H(R-|x|)\Big) \,.
\end{align*}

Thus by the Fourier inversion formula, 
\begin{align*}\frac{\sin(R|\xi|)}{|\xi|}&=d_n\, \Big(\frac1{R}\,\frac{\partial}{\partial
R}\Big)^{(n-2)/2}\Big( \int_{\R^n} 
\frac1{\sqrt{R^2-|x|^2}} \,H(R-|x|) \,e^{-ix\cdot \xi}\,dx\Big) \\
&= d_n\,\Big(\frac1{R}\,\frac{\partial}{\partial
R}\Big)^{(n-2)/2} \Big( \frac{1}{\,v_n }\int_{  B(0,R)}\frac1{\sqrt{R^2-|x|^2}}\,
e^{-ix\cdot \xi}\,d x\Big)\,,
\end{align*}

The constant $d_n$ is readily determined as in \cite{WT},  and we have finished.

 \end{document}